\documentclass[a4paper,reqno]{amsart}

\def\@logofont{\footnotesize}
\def\@setaddresses{\par
  \nobreak \begingroup
  \footnotesize
  \def\author##1{\nobreak\addvspace\bigskipamount}%
  \def\\{\par\nobreak}%
  \interlinepenalty\@M
  \def\address##1##2{\begingroup
    \par\addvspace\bigskipamount\indent
    \@ifnotempty{##1}{(\ignorespaces##1\unskip) }%
    {\scshape\ignorespaces##2}\par\endgroup}%
  \def\curraddr##1##2{\begingroup
    \@ifnotempty{##2}{\nobreak\indent\curraddrname
      \@ifnotempty{##1}{, \ignorespaces##1\unskip}\/:\space
      ##2\par}\endgroup}%
  \def\email##1##2{\begingroup
    \@ifnotempty{##2}{\nobreak\indent\emailaddrname
      \@ifnotempty{##1}{, \ignorespaces##1\unskip}\/:\space
      \ttfamily##2\par}\endgroup}%
  \def\urladdr##1##2{\begingroup
    \def~{\char`\~}%
    \@ifnotempty{##2}{\nobreak\indent\urladdrname
      \@ifnotempty{##1}{, \ignorespaces##1\unskip}\/:\space
      \ttfamily##2\par}\endgroup}%
  \addresses
  \endgroup
}
\renewcommand*\subjclass[2][2010]{%
  \def\@subjclass{#2}%
  \@ifundefined{subjclassname@#1}{%
    \ClassWarning{\@classname}{Unknown edition (#1) of Mathematics
      Subject Classification; using '2000'.}%
  }{%
    \@xp\let\@xp\subjclassname\csname subjclassname@#1\endcsname
  }%
}

\usepackage{graphicx,amsmath,amssymb}
\allowdisplaybreaks

\theoremstyle{plain}
\newtheorem{theorem}{Theorem}[section]

\newtheorem{lemma}[theorem]{Lemma}
\newtheorem{corollary}[theorem]{Corollary}
\newtheorem{conjecture}[theorem]{Conjecture}

\theoremstyle{definition}
\newtheorem{definition}[theorem]{Definition}
\newtheorem{example}[theorem]{Example}

\newtheorem{remark}[theorem]{Remark}

\begin{document}

\title{On local matching property in groups and vector spaces}

\author[M. Aliabadi]{Mohsen Aliabadi}
\address{Mohsen Aliabadi \\ 
Department of Mathematics, Statistics, and Computer Science\\ University of  Illinois\\  851 S. Morgan St, Chicago, IL 60607, USA}
\email{maliab2@uic.edu}

\author[M. V. Janardhanan]{Mano Vikash Janardhanan}
\address{Mano Vikash Janardhanan\\ 
Department of Mathematics, Statistics, and Computer Science\\ University of  Illinois\\  851 S. Morgan St, Chicago, IL 60607, USA}
\email{mjana2@uic.edu}

%\subjclass{Primary: 05D15; Secondary: 11B75, 20D60, 20F99, 12F99.}

\begin{abstract}
In this paper, we define locally matchable subsets of a group which is derived from the concept of matchings in groups and used as a tool to give alternative proofs for existing results in matching theory. We also give the linear analogue of local  matching property for subspaces in a field extension. 
Our tools mix additive number theory, combinatorics and algebra.
\end{abstract}

\maketitle

\section{Introduction}
The notion of {\it matchings} in groups was used to study an old problem of Wakeford concerning canonical forms for symmetric 
tensors \cite{11}.  Losonczy in \cite{8} introduced matchings in order to generalize a geometric property of lattices in Euclidean space.
A matching in an abelian group $(G,+)$ is a bijection $f:A\to B$, where $A$, $B$ are finite subsets of $G$ such that $0\not\in B$, fulfilling $a+f(a)\not\in A$, for all $a\in A$, and $G$ is said to have the {\it matching property}, if matching always exist, as long as $A$ and $B$ are finite of the same cardinality. This topic has found some interest in literature, for example, in 2006 Eliahou-Lecouvey have generalized Losonczy's results to arbitrary groups \cite{1}, and in 2010 Eliahou-Lecouvey went over to subspaces in field extensions \cite{6}.

The subject of the present paper is to consider {\it local matchings}: given a proper subgroup $H<G$ such that $H\cap B\neq \emptyset$ and $a+H\subseteq A$, for some $a\in A$, there is a bijection $f:A'\to H\cap B$, for some $A'\subseteq A$, such that $a+f(a)\not\in A$, for all $a\in A'$. In this case, $A$ is called to be {\it locally matched} to $B$. Any matching  being a local matching, it is natural to ask whether conversely a local matching property implies the matching property. The answer to this question is ``Yes" and we will use this result to give an alternative proof for Losonczy's main result in \cite{8}. Moreover, these questions are also discussed in the context of subspaces in field extensions.

The purpose of this paper is to find the relations between local  matching property and matching property in groups and vector spaces to give alternative proofs for existing results on matching property for groups and also its linear analogue.  Section 2 is devoted to the results proved  on matching in groups and vector spaces and also some tools of additive number theory required to prove our main results. 
In  Section 3, we will show the equivalence between matching and local  matching for subsets of a group. %which leads to some generalizations of author's results in \cite{}. 
Section 4  concerned with the linear analogue  of one  of Losonzy's results on matchings for cyclic groups.  %Furthermore, we will present a dimension criteria for the primitive subspaces related to our results in Theorem 4.1.%would be discussed in Theorem .
\,Finally, in Section 5, we   show that %if $A$ is matched to $B$, then $A$ it is locally matched to $B$ in the sense of vector spaces. However, the  converse is still not clear in the general case. We will see in Theorem 5.2 that 
\,for vector spaces in a field extension whose 
algebraic elements are separable, the linear local matching property implies the matching property.% As an application, we give an alternative proof for authors' main result in \cite{6}. % which is a generalization of auther's results in \cite{6} on linear matching property in field extensions. %The interested reader is encouraged  to investigate this property for other families of field extensions.
\section{Preliminaries}
%The following theorems are known about matching for groups and vector spaces. As we already mentioned, our goal in this paper is to prove their  more general cases.
%\begin{theorem}\label{t2.1}
 %An abelian group $G$ has the matching property if and only if $G$ is torsion-free or cyclic of 
%prime order \cite{8}.
%\end{theorem}
%\begin{theorem}\label{t2.2}
% Let  $G$ be a non-trivial finite cyclic group. Suppose we are given non-empty subsets $A$ and $B$ 
%of $G$  such that $\# A=\# B$ and every element of $B$ is a generator of $G$. Then there exists at least one matching 
%from $A$ to $B$ \cite{8}.
%\end{theorem}
First, we define the matching property for subspaces in a field extension. 
 Let $K\subset L$ be a field extension and $A$ and $B$ be $n$-dimensional $K$-subspaces of 
the field extension $L$. Let $\mathcal{A}=\{a_1,\ldots,a_n\}$ and $\mathcal B=\{b_1,\ldots,b_n\}$ be bases of $A$ and 
$B$, respectively. It is said that $\mathcal A$ is \textit{matched} to  $\mathcal B$ if
\[
a_ib\in A \;\Rightarrow \; b\in \langle b_1,\ldots,\hat{b}_i,\ldots,b_n\rangle,
\]
for all $b\in B$ and $i=1,\ldots,n$, where $\langle b_1,\ldots,\hat{b}_i,\ldots,b_n\rangle$ is the hyperplane of $B$ 
spanned by the set $\mathcal B\setminus \{b_i\}$; moreover, it is said that $A$ is \textit{matched } to $B$ if every 
basis $\mathcal A$ of $A$ can be matched to a basis $\mathcal B$ of $B$. As it is seen, the matchable bases are defined 
in a natural way based on the definition of matching in a group. Indeed, we can consider $\mathcal{A}$ and 
$\mathcal{B}$ as subsets of the multiplicative group $L^*$ and so the bijection $a_i\mapsto b_i$ is a matching in the group 
setting sense. %\\
It's said $L$ has the \textit{linear matching property} if, for every $n\geq1$
and every $n$-dimensional subspaces $A$ and $B$ of $L$ with $1\not\in B$, the subspaces $A$ is matched with $B$.  A {\it strong matching} from $A$ to $B$ is a linear isomorphism $\varphi:A\to B$ such that any basis $\mathcal{A}$ of $A$ is matched to the basis $\varphi(\mathcal{A})$ of $B$. It is proved that there is a strong  matching from $A$ to $B$ if and only if $AB\cap A=\{0\}$. In this case, any isomorphism $\varphi:A\to B$ is a strong matching \cite{6}.\\
%\\[2mm]
Now, our definition for matchable subsets of two matchable bases:%\\%[5mm]
\begin{definition}
 Let $\tilde{A}$ and $\tilde{B}$ be two non-zero $m$-dimensional  $K$-subspaces of $A$ and $B$, respectively. We say that 
$\tilde{A}$ is {\it $A$-matched} to $\tilde{B}$, if for any basis $\tilde{\mathcal{A}}=\{a_1,\ldots,a_m\}$ of $\tilde{A}$, 
there exists a basis $\tilde{\mathcal{B}}=\{b_1,\ldots,b_m\}$ of $\tilde{B}$ for which $a_ib_i\not\in A$, for 
$i=1,\ldots,m$. In this case, it is also said that $\tilde{\mathcal{A}}$ is {\it $A$-matched} to $\tilde{\mathcal{B}}$.
\end{definition}
The following is the linear analogue of locally matchable subsets for the vector spaces in a field extension.
\begin{definition}
 Let $K\subset L$ be a field extension and $A$, $B$ be two $n$-dimensional $K$-subspaces of 
$L$.  We say that $A$ is {\it  locally matched} to $B$ if for any intermediate subfield $K\subset H\subsetneqq L$ with $H\cap B\neq\{0\}$ and $aH\subseteq 
A$, for some $a\in A$, one can find a subspace $\tilde{A}$ of $A$ such that $\tilde{A}$ is $A$-matched to $H\cap B$.
\end{definition}
\begin{definition}
We say that $K\subset L$ has the {\it linear local matching property} if, for every $n\geq1$ and every $n$-dimensional subspaces $A$ and $B$ of $L$ with $1\not\in B$, the subspace $A$ is locally matched to $B$.
\end{definition}
The following theorem is a dimension criteria for matchable bases \cite[Proposition 3.1]{6}. This will be used as a tool to prove Theorem \ref{t4.2}. For more results on linear version of matchings  see  \cite{3}.
\begin{theorem}\label{t2.4}
 Let $K\subset L$ be a field extension and $A$ and $B$ be two $n$-dimensional $K$-subspaces of 
$L$. Suppose that $\mathcal A=\{a_1,\ldots,a_n\}$ is a basis of $A$. Then $\mathcal A$ can be matched to a basis of 
$B$ if and only if, for all $J\subseteq\{1,\ldots,n\}$, we have:
\[
\dim_K\underset{i\in J}{\bigcap}\left(a_i^{-1}A\cap B\right)\leq n-\# J.
\]
\end{theorem}
%\textbf{Theorem 1.9.} Let $K\subseteq L$ be a field extension.  Then $K\subseteq L$ has the linear matching property 
%if and only if $L$ contains no proper finite-dimensional extension over $K$ \cite{4}.
%See \cite{6} for more details.
%One of the main result in \cite{6} is that a field extension $K\subset L$ has the linear matching property if and only  if there are no trivial finite intermediate extension $K\subset M\subset L$. We would like to mention that, although the statement of \cite[Theorem 2.6]{6} is slightly different and assumes that the  extension is either purely transcendental or finite of prime degree, what they actually use in their proof is that there are no nontrivial finite intermediate extensions, which is a weaker condition. (See also \cite{1}.)

{The following theorem gives the necessary and sufficient condition for a field extension to have the linear matching property.}
\begin{theorem}\label{t2.5}
Let $K\subset L$ be a field extension. Then, $L$ has the linear matching property if and only if $L$ contains no proper finite dimensional extension over $K$. \cite[Theorem 5.2]{6}. 
\end{theorem}

%\section{Some results from Algebra, additive number theory and combinatorics}
For proving our main results, we shall need the following theorem from \cite[page 116, Theorem 4.3]{9}.
\begin{theorem}[Kneser]
  If $C=A+B$, where $A$ and $B$ are finite subsets of an abelian group $G$, then
\[
\# C\geq \# A+\# B-\# H,
\]
where $H$ is the subgroup $H=\{g\in G:\; C+g=C\}$.
\end{theorem}
 %{We shall use the following version of marriage theorem of Hall to prove Lemma \ref{l3.2}. Recall that, given a collection $\mathcal{E} =\{E_1,\ldots,E_n\}$ of subsets of a set $E$, a {\it transversal} for $E$ is a set $\{x_1,\ldots,x_n\}$ of pairwise distinct elements of $E$ with property $x_i\in E_i$, for all $i\in\{1,\ldots,n\}$.
% \begin{theorem}[Hall \cite{7}]
% Let $E$ be a set and $\mathcal{E}=\{E_1,\ldots,E_n\}$ be a family of finite subsets of $E$. Then, there exists a transversal for $\mathcal{E}$ if and only if
% \[\#\cup E_i\geq\#J,\]
% for any non-empty subset $J\subseteq\{1,\ldots,n\}$.
% \end{theorem}
 See \cite{5} for more details regarding the following theorem which is the linear analogue of Kneser's theorem.
\begin{theorem}
 Let $K\subset L$ be a field extension in which every algebraic element of $L$ is separable 
 over $K$. Let $A, B\subset L$ be non-zero finite-dimensional $K$-subspaces of $L$ and $H$ be the stabilizer of 
 $\langle AB\rangle$, i.e. $H=\{x\in L;\, x\langle AB\rangle\subseteq\langle AB\rangle\}$. Then
 \[
 \dim_K\langle AB\rangle\geq\dim_KA+\dim_KB-\dim_KH.
 \]
\end{theorem}
We remark that in the above theorem, we denote by $\langle AB\rangle$  the $K$-subspace of $L$ generated by the Minkowski product  $AB$ which is defined as $AB:=\{ab;\, a\in A, b\in B\}$.

\section{Local  matching property for groups}
The following theorem  shows that local  matching property is equivalent to matching property  in  abelian groups. %Note that it is obvious that matchability implies local  matchability. Then, we only show its converse case. 
 {The main idea of our proof is obtained from the Loconczy paper \cite[Theorem 3.1]{8} and Eliahou-Lecouvey paper \cite[Theorem 3.3]{1}.}
\begin{theorem}\label{t3.1}
 Let $G$ be an additive abelian group and $A$, $B$ be non-empty finite subsets of $G$ satisfying 
the conditions $\# A=\# B$ and $0\not\in B$. %Assume that for any subgroup $H$ of $G$ with $H\cap B\neq \emptyset$ 
%and $a+H\subseteq A$ for some $a\in A$, there exists an $A$-matching from a subset of $A$ to $H\cap B$. Then, there is a matching from $A$ to $B$.
 If $A$  is locally matched to $B$, then $A$ is matched to $B$.
\end{theorem}
We shall need the following lemma to prove Theorem \ref{t3.1}.
\begin{lemma}\label{l3.2}
Let $G$, $A$ and $B$ be as Theorem \ref{t3.1}. For any non-empty subset $S$ of $A$, assume that $\# S\leq \#(B\setminus U)$, where $U=\{b\in B; s+b\in A,$ for any $s\in S\}$. Then, there is a matching from $A$ to $B$.
\end{lemma}
\begin{proof}
Assume that $A=\{a_1,\ldots,a_n\}$ and define $S_J=\{a_i;\, i\in J\}$, for any $J\subseteq \{1,\ldots,n\}$. Set $U_J=\{b\in B;\, s+b\in A,$ for any $s\in S_J\}$. Clearly, $U_J=\underset{i\in J}{\bigcap}U_{\{i\}}$. Consider the collection $\mathcal{E}=\{B\setminus U_{\{1\}},\ldots,B\setminus U_{\{n\}}\}$. We have 
\[\#\underset{i\in J}{\bigcup}\left(B\setminus U_{\{i\}}\right)=\# \left(B\setminus \underset{i\in J}{\bigcap}U_{\{i\}}\right)= \#\left(B\setminus U_{J}\right)\geq \# S_J=\# J,\]
for any $J\subseteq\{1,\ldots,n\}$. Then by Hall marriage Theorem \cite[Theorem 2]{7}, one can find a transversal $(b_1,\ldots,b_n)\in\mathcal{E}$. The mapping $a_i\mapsto b_i$ is a matching from $A$ to $B$.
\end{proof}
\noindent{\it Proof of Theorem \ref{t3.1}.} {We remark that for the case $G$ is a finite group and $A=G\setminus\{0\}$, we have the matching $\begin{array}{rl}\\f:A&\to B\\a&\mapsto -a\end{array} $. Thus, we may assume that $A\neq G\setminus\{0\}$. }
 Suppose there is no matching from $A$ to $B$. We are going to reach a contradiction. Using Lemma \ref{l3.2}, there exists a non-empty finite subset $S$ of $A$ such that $\# (B\setminus U)<\# S$, where $U=\{b\in B:\; s+b\in 
A, $ for any $s\in S\}$. Let $\# A=\# B=n$, then $\# U+\# S>n$.
Set $U_0=U\cup\{0\}$. Using Kneser's Theorem one can find the subgroup $H$ of $G$ such that 
\begin{equation}\label{eq1}
\# (U_0+S)\geq\# U_0+\# S-\#  H,
\end{equation}
where $H=\{g\in G:\; g + U_0 +S=U_0+S\}$. Applying  Kneser's Theorem for $U'=H\cup U$ and $S$, we can find the subgroup $H'$ of $G$ for which 
\begin{equation}\label{eq2}
\#(U'+S)\geq \# U' +\# S- \# H',
\end{equation}
 where  $H'=\{g\in G:\; g+U'+S=U'+S\}$. We claim that $H=H'$ and to prove this, it suffices to show that $U'+S=U_0+S$. We have
 \begin{align}\label{eq3}
 U'+S&=(H\cup U)+S=(H+S)\cup (U_0+S)\nonumber\\
 &=(H+S)\cup(U_0+S+H)\nonumber\\
 &= H+(S\cup(U_0+S))\nonumber\\
 &=H+(U_0+S)=U_0+S.
 \end{align}
Then  $H=H'$ and it follows from \eqref{eq2} that
\begin{equation}\label{eq4}
\#  (U_0+S)\geq \# U' +\# S-\# H.
\end{equation}
Using \eqref{eq3}, \eqref{eq4} we obtain
\begin{eqnarray}\label{eq5}
\# (U_0+S)&=& \#(U'+S)\nonumber\\
&=&\# U' + \# S- \# H \nonumber\\
&=& \# (H\cup U) + \# S -\# H\nonumber\\
&=&\# H+\# U -\#(H\cap U)+\# S- \# H\nonumber\\
&=& \# U+ \# S -\# (H\cap U).
\end{eqnarray}
As $U_0+S=S\cup (S+U)$, \eqref{eq5} implies
\begin{equation}\label{eq6}
\#(S\cup (S+U))\geq \# U+\# S-\# (H\cap U).
\end{equation}
Now, we have two cases for $H\cap U$.
\begin{enumerate}
\item
If $H\cap U$ is empty, then by \eqref{eq6} we conclude that $\#(S\cup (S+U))\geq n$. On the other hand $S\cup (S+U)$  is a subset of $A$. We would have $\# A>n$, which contradicts $\# A=n$ above.
\item
If  $H\cap U$ is non-empty,  so is $H\cap B$. Also, if $s\in S\subseteq A$, then according to  the definition of $H$, $s+H\subseteq U_0+S+H=U_0+S\subseteq A$. As $A$ is locally matched to $B$, then there is  a subset $\tilde{A}$ of $A$ and a bijection $f:\tilde{A}\to H\cap B$ such that $a+f(a)\not\in A$, for any $a\in\tilde{A}$. We claim that $f^{-1}(H\cap U)\cap (U_0+S)$ is empty. If not and $a\in f^{-1}(H\cap U)\cap (U_0+S)$, then  $a+f(a)\in (U_0+S)+H$ as $ a \in U_0+S$ and $f(a) \in H\cap U\subseteq H$. Since $U_0+S\subseteq A$, then $a+f(a)\in A$ which is a contradiction. Therefore $f^{-1}(H\cap U)\cap(U_0+S)$ is empty. As the sets $f^{-1}(H\cap U)$ and $U_0+S$ are both subsets of $A$ and have nothing in common, then $\# f^{-1}(H\cap U)+\#(U_0+S)\leq n$. Thus $\#(H\cap U)+\#(U_0+S)\leq n$ and this tells us $\#(H\cap  U)+\#(S\cup(S+U))\leq n$. Next, using \eqref{eq6} yields that $\# U+\# S\leq n$ which is a contradiction. 
\end{enumerate}
Therefore in both cases we extract contradictions. Then there is a matching from $A$ to $B$.  \hfill $\Box$
\begin{remark}
{Note that in the second case above, $H\neq G$. We argue this in two cases:
\begin{enumerate}
\item
Suppose that $0\in A$. Assume to the contrary $G=H$. Then we have 
\[\#G=\#H\leq \# (H+U_0+S)=\# (U_0+S)\leq\# A.\]
Then $\#G=\#A=\#B$. This contradicts $0\not\in B$.
\item
Suppose that $0\not\in A$. Assume to the contrary $G=H$. Then we have
\[\#G=\#H\leq \# (H+U_0+S)=\# (U_0+S)\leq\# (A\cup\{0\}).\]
Then $G=A\cup \{0\}$ and this contradicts $A\neq G\setminus \{0\}$.
\end{enumerate}
}
\end{remark}
\begin{example}
Assume that $G=\mathbb{Z}/8\mathbb{Z}$, $A,B\subseteq G$ with $A=\{0,2,6\}$ and $B=\{1,3,4\}$. The only non-trivial subgroup of $G$ which satisfies the condition $a+H\subseteq A$ is $H=\{0,4\}$. Note that here $a=2$. If $A'=\{0\}\subseteq A$, then for the bijection $f:A'\to H\cap B$ defined as $f(0)=4$ we have $0+f(0)=4\not\in A$. Then $A$ is locally matched to $B$ and using Theorem \ref{t3.1}, $A$ is matched to $B$.
\end{example}
%\textbf{Corollary 3.2.} 
Using Theorem \ref{t3.1}, we give an alternative proof to the following result of Losonzcy \cite[Theorem 3.1]{8}.
\begin{theorem}
An abelian group $G$ has the matching property if and only if it is torsion-free or cyclic of prime order.
\end{theorem}
\begin{proof}
Assume that $G$ is either torsion-free or cyclic of prime order. Then $G$ has no non-trivial subgroup of finite order. This means if $A,B\subset G$ with $\# A=\# B$ and $0\not \in B$, then $A$ is locally matched to $B$ (because in this case $H=\{0\}$ is the only proper subgroup of $G$. But $H\cap B=\emptyset$.) Using Theorem \ref{t3.1}  yields that $A$ is matched to $B$ and so $G$ has matching property. \\
Conversely, assume that $G$ is neither torsion-free nor cyclic of prime order. Then it has a non-trivial finite subgroup $H$. Choose $g\in G\setminus H$, set $A=H$ and $B=H\cup \{g\}\setminus \{0\}$. Clearly, $H\cap B\neq \emptyset$ and $a+H\subseteq A$ for some $a\in A$ (Indeed for any $a\in A$). If $A$ is locally matched to $B$, then one can find an $A$-matching $f$ from a subset $A_0$ of $A$ to $H\cap B$. But if $a\in A_0$, then $a+f(a)\in H+(H\cap B)=H+(H\setminus\{0\})=H=A$, which is a contradiction. Then $A$ is not locally matched to $B$ and so by Theorem \ref{t3.1}, $A$ is not matched to $B$. Therefore $G$ has no  matching property.
\end{proof}
\begin{corollary}
Let $G$, $A$ and $B$ be as Theorem \ref{t3.1} and $\# A=\# B=n>1$. Denote by $n(G)$ the smallest cardinality of a non-zero subgroup of $G$. If $n<n(G)$, then $A$ is matched to $B$.
\end{corollary}
\begin{proof}
Since $n<n(G)$, then it is clear that $A$ is locally matched to $B$. Using Theorem \ref{t3.1} yields $A$ is matched to $B$.
\end{proof}
\section{The linear analogue of Losonzcy's result on matchable subsets}
%The following is the  linear version of Theorem \ref{t2.2}   which investigates the matchable subspaces in a simple field extension. Here, 
In this section, we formulate and prove the linear analogue of the following theorem of Losonzcy proven in \cite{8} which basically investigates the matchable subspaces in a simple field extension.
\begin{theorem}
Let $G$ be a non-trivial finite cyclic group such that $\# A=\#B$ and every element of $B$ is a generator of $G$. Then there exists at least one matching from $A$ to $B$.
\end{theorem}
We say that $K\subset L$ is a simple field extension if $L=K(\alpha)$, for some $\alpha\in L$. Also, if $B$ is a $K$-subspace of $L$ such that $K(b)=L$, for any $b\in B\setminus \{0\}$, we say that $B$ is a primitive $K$-subspace of $L$.  The main ingredient in our proof is the linear version of Kneser's theorem.
\begin{theorem}\label{t4.2}
 Let $K\subset L$ be  separable field extension and $A$ and $B$ be two 
$n$-dimensional $K$-subspaces in $L$ with $n\geq1$ and $B$ is a primitive $K$-subspace of $L$.  Then $A$ is matched with 
$B$.
\end{theorem}
\begin{proof}
 Assume that  $A$ is not  matched  to $B$. Using Theorem \ref{t2.4}, one can find  $J\subseteq\{1,\ldots,n\}$ 
and a basis $\mathcal A=\{a_1,\ldots,a_n\}$ of $A$ such that
\begin{equation}\label{eqn7}
\underset{i\in J}{\bigcap}\left(a_i^{-1}A\cap B\right)>n-\# J.
\end{equation}
Set  $S=\langle a_i:\; i\in J\rangle$ the $K$-subspace of $A$ spanned by $a_i$'s, $i\in J$, $U=\underset{i\in J}{\bigcap}\left(a_i^{-1}A\cap B\right)$ and  $U_0=U\cup\{1\}$. Now, by Theorem 2.7 one can find a subfield $H$ of $L$ such that
\begin{eqnarray*}
\dim_K\langle U_0S\rangle &\geq& \dim_KU_0+\dim_KS-\dim_KH,
\end{eqnarray*}
where $H$ is the stabilizer of $\langle U_0S\rangle$, i.e. $H=\{x\in L:\; x\langle U_0S\rangle\subseteq \langle U_0S\rangle\}$. Define  $U'=\langle H\cup U\rangle$. Using Theorem 2.7 again, we have 
\[\dim_K\langle U'S\rangle \geq \dim_K \langle U'\rangle +\dim_K\langle S\rangle -\dim_KH',\]
where $H'$  is the stabilizer of  $\langle U'S\rangle$. Next, we have
\begin{eqnarray}\label{eqn8}
\langle U'S\rangle&=&\langle (H\cup U)S\rangle =\langle HS\cup U_0S\rangle\nonumber\\
&=&\langle HS\cup HU_0S \rangle =H\langle S\cup U_0S\rangle\nonumber\\
&=&H\langle U_0S\rangle =\langle U_0S\rangle.
\end{eqnarray}
This follows $H=H'$ and  then
\begin{equation}\label{eqn9}
\dim_K\langle U'S\rangle \geq\dim_K\langle U'\rangle +\dim_KS-\dim_KH.
\end{equation}
Using \eqref{eqn8} and \eqref{eqn9}, we have 
\begin{eqnarray}\label{eqn10}
\dim_K\langle U_0S\rangle&\geq& \dim_KU'+\dim_KS-\dim_KH\nonumber\\
&=& \dim_K\langle H\cup U\rangle+\dim_KS-\dim_KH.
\end{eqnarray}
Using \eqref{eqn10}, the fact that $\langle U_0S\rangle=\langle S\cup SU\rangle$ and the inclusion-exclusion principle for vector spaces we have:
\begin{eqnarray}\label{eqn11}
\dim_K\langle S\cup SU\rangle&\geq& \dim_K \langle H\cup U\rangle+\dim_KS-\dim_KH\nonumber\\
&=& \dim_K H+\dim_K U -\dim_K(H\cap U)+\dim_KS-\dim_KH\nonumber\\
&=&\dim_K U+\dim_KS-\dim_K(H\cap U).
\end{eqnarray}
Now, we have two cases for the subspace $H\cap U$.
\begin{enumerate}
\item
If $H\cap U=\{0\}$, then  \eqref{eqn7} and \eqref{eqn11} imply  $\dim_K \langle S\cup SU \rangle\geq n$ and this is impossible as   
$S\cup SU\subseteq A$ and $\dim_KA=n$.%\\
\item
If $H\cap U\neq\{0\}$, then $H\cap B\neq\{0\}$ and since $B$ is a primitive subspace of $L$, then $H=L$. By the definition of $U$ and $S$, $HUS\subseteq A$ and  this follows $LUS\subseteq A$ and so $A=L$. Then $B=L$ as  $\dim_K A=\dim_K B$ and this 
means $K\subseteq B$. Therefore if $a\in K\setminus\{0\}$, then $K=K(a)=L$, which is impossible. 
\end{enumerate}
In both cases, we extract contradictions. Then $A$ is matched to $B$.
\end{proof}
\section{Local  matching property for subspaces in a field extension}
The following theorem shows that the linear local  matching property implies  the linear matching property  for subspaces of a field extension whose algebraic  elements are separable. Note that this result can probably be reformulated for any field extension $K\subset L$ without any condition on separability.
\begin{theorem}\label{t5.1}
 Let $K\subset L$ be a field extension in which every algebraic element of $L$ is separable over 
$K$. Let $A, B\subset L$ be two non-zero $n$-dimensional $K$-subspaces with $1\not\in B$. If $A$ is locally matched 
to $B$, then   $A$ is matched to $B$.
\end{theorem}
\begin{proof}
 Assume to the contrary $A$ is not matched to $B$.  Then, by Theorem 2.3 there exist a basis 
$\mathcal{A}=\{a_1,\ldots,a_n\}$ of $A$ and $J\subseteq\{1,\ldots,n\}$ such that $\dim_K\underset{i\in J}{\bigcap}\left(a_i^{-1}A\cap B\right)>n-\# J$. Set $S=\langle a_i:\; i\in J\rangle$ as a $K$-subspace of $A$, 
$U=\underset{i\in J}{\bigcap}(a_i^{-1}A\cap B)$ and $U_0=\langle U\cup \{1\}\rangle$. 
Using Theorem 2.7 there exists an intermediate subfield $H$ of $K\subset L$ such that 
\begin{equation}\label{eqn13}
\dim_K\langle U_0S\rangle\geq\dim_K U_0+\dim_KS-\dim_KH,
\end{equation}
where $H$ is the stabilizer of $\langle U_0S\rangle$. Define  $U'=H\cup U$. Reusing Theorem 2.7, one can find an intermediate subfield $H'$ of $K\subset L$ for which 
\begin{equation}\label{eqn14}
\dim_K \langle U'S\rangle \geq \dim_K \langle U'\rangle+\dim_KS-\dim_KH',
\end{equation}
where $H'$ is the stabilizer  of  $\langle U'S\rangle$. The following computations show that $\langle  U'S\rangle=\langle U_0S\rangle$;
 \begin{eqnarray}\label{eqn15}
 \langle U'S\rangle&=&\langle (H\cup U)S\rangle =\langle HS\rangle\cup\langle U_0S\rangle\nonumber\\
 &=&\langle HS\rangle\cup\langle U_0SH\rangle=H\langle S\cup U_0S\rangle\nonumber\\
 &=&H\langle U_0S\rangle=\langle U_0S\rangle.
 \end{eqnarray}
 Then, the stabilizers of these two subspaces must be the same, i.e. $H=H'$. Then we would  have 
 \begin{equation}\label{eqn16}
  \dim_K\langle U'S\rangle\geq\dim_K\langle   U'\rangle+\dim_KS-\dim_KH.
 \end{equation}
Bearing \eqref{eqn14} and \eqref{eqn15} in mind and using the inclusion-exclusion principle for vector spaces we obtain: 
\begin{eqnarray}\label{eqn17}
\dim_K\langle U_0S\rangle&=& \dim_K\langle U'S\rangle\nonumber\\
&\geq&\dim_K\langle U'\rangle +\dim_KS-\dim_KH\nonumber\\
&=&\dim_K\langle H\cup U\rangle+\dim_KS-\dim_KH\nonumber\\
&=& \dim_KH+\dim_KU-\dim_K(H\cap U)+\dim_KS-\dim_KH\nonumber\\
&=&\dim_KU+\dim_KS-\dim_K(H\cap U).
 \end{eqnarray} 
 Now, we have two cases for $H\cap U$.
 \begin{enumerate}
 \item
If $H\cap U=\{0\}$, then $\dim_K\langle S\cup SU\rangle>n$. On the other hand since $S\cup SU\subseteq A$, we would 
have $\dim_KA>n$, contradicting our assumption $\dim_KA=n$.
\item
If  $H\cap U$ is a non-zero vector space, then $H\cap B$ is non-zero. 
It is clear that $aH\subseteq A$, for some $a\in A$ (Indeed, $USH\subseteq A$).
 Since $A$ is locally matched to $B$, one can find a subspace $\tilde{A}$ of $A$ such that $\tilde{A}$ is $A$-matched 
 to $H\cap B$. Let $\tilde{A}\cap \langle U_0S\rangle\neq\{0\}$ and choose a non-zero element $a$ of it.  We extend $\{a\}$ to a basis $\{a, 
 a_2,\ldots, a_m\}$ for  $\tilde{A}$. Then,  there exists a basis $\{b, b_2, \ldots, b_m\}$ of $H\cap B$ such that $ab 
 \not\in A$ and $a_ib_i \not\in A$, where $2 \leq i \leq m$, as $A$ is locally matched to $B$. But, we have $ab \in \langle U_0S\rangle H=\langle U_0S\rangle \subseteq A$, which contradicts the case $\tilde{A}$ is $A$-matched to $H\cap B$. So $\tilde{A}\cap U_0S=\{0\}$. Then, $\dim_K \tilde{A} + \dim_K\langle U_0S\rangle \leq n$.
This yields $\dim_K\langle H\cap U\rangle+\dim_K\langle (U\cup\{1\})S\rangle \leq n$. This follows $\dim_K\langle H\cap 
U\rangle+\dim_K\langle S\cup SU\rangle\leq n$. So, by \eqref{eqn17} we have $\dim_KU+\dim_KS\leq n$, which is impossible.
 \end{enumerate}
 Then in both cases, we extract contradictions and so  $A$ is matched to $B$. 
\end{proof}
\begin{remark}
Note that in the second  case above, $H\neq L$. We justify this  as follows; assume to the contrary $H=L$. Then we have 
\[
[L:K]=[H:K]\leq \dim_KH\langle U_0S\rangle=\dim_K\langle  U_0S\rangle=\dim_K\langle US \cup S\rangle\leq\dim_KA=\dim_KB.
\]
 Then $B=L$ and this contradicts $1\not\in B$.
\end{remark}
\begin{example}
Assume that $K=\mathbb{Q}$ and $L=\mathbb{Q}(\sqrt[35]{2})$.  Let $A$ be the subspace of those elements of $\mathbb{Q}(\sqrt[5]{2})$ whose trace over $\mathbb{Q}$ is zero. Then $\dim_kA=4$. Let $V$ be the subspace of those elements of $\mathbb{Q}(\sqrt[7]{2})$ whose trace over $\mathbb{Q}$ is zero and  take a four-dimensional subspace $B$ of $V$.
Clearly, for every non-trivial intermediate subfield $H$ of $K\subseteq L$, $[H:K]\geq5$ and then  $aH\nsubseteq A$, for any $a\in A$. Thus $A$ is locally matched to $B$. Using Theorem \ref{t5.1}, $A$ is matched to $B$.
\end{example}
\begin{corollary}
Let $K\subset L$ be a field extension in which every algebraic element of $L$ is separable over $K$. If $K\subset L$  has the local linear matching property, then it possesses the linear  matching property.
\end{corollary}
As we mentioned, in Theorem \ref{t5.1}  the extension $K\subset L$ is assumed to have all its algebraic elements separable. Are these results valid without this hypothesis? We conjecture that this is the case.
\begin{conjecture}
Let  $K\subset L$ be a field extension. Then $K\subset L$ possesses the linear matching property if it possesses the local linear matching property.
\end{conjecture}
Using Theorems \ref{t5.1}, we give a short proof of a special case of Theorem \ref{t2.5}.\\
%an alternative proof to the authors' main theorem in \cite{6} which states that the necessary and sufficient condition for a field extension having linear matching property is not having any proper intermediate field with finite degree. However, our proof works just in the case that every algebraic element is assumed to be separable. Note that we use this fact that matchability and local matchability are  equivalent in such a field extension. 
Let $K\subset L$ be a field extension whose algebraic elements are separable and has no proper intermediate field with a finite degree. If $A$ and $B$ are two $n$-dimensional $K$-subspaces of $L$ with $n\geq1$ and $1\not\in B$, clearly $A$  is locally matched to $B$ and then $A$ is matched to $B$. This means $K\subset L$ has the linear matching property.
\begin{remark}
That the linear matching property implies the local linear matching property is immediate from Theorem \ref{t2.5}. However, If $A$ is matched to $B$, in the field extension setting sense, whether $A$ is locally matched to $B$ is still unsolved. This is valid is some specific cases. For example, if there is a strong matching from $A$ to $B$, one can prove that $A$ is locally matched to $B$. Further investigations along those lines  could prove to be worthwhile.
\end{remark}
%Conversely, assume that $K\subseteq L$ has the proper intermediate field $H$ of finite degree. Set $A=H$, choose $x\in L\setminus H$ and define $B=\langle H\cup \{x\}\setminus K\rangle$. Clearly $\dim_KA=\dim_KB$, $H\cap B\neq \{0\}$ and $aH\subseteq H=A$, for some $a\in A$. (Indeed for any $a\in A$). If $A$ is locally matched to $B$, then  a $K$-subspace $\tilde{A}$ of $A$ is $A$-matched to $H\cap B$. If $\{a_1,\ldots,a_m\}$ is a basis for $\tilde{A}$, then there must be a basis $\{b_1,\ldots,b_m\}$ of $H\cap B$ for which $a_ib_i\not\in A$, for $1\leq i\leq m$. But  this is impossible because $a_ib_i\in A(H\cap B)\subseteq H\langle H\setminus K\rangle \subseteq H=A$. Then, $A$ is not locally matched to $B$ and so $A$ is not matched to $B$. Thus $K\subseteq L$ does not have the linear matching property.
\section*{Acknowledgement}
We are grateful to Professor Shmuel Friedland for his useful suggestions and comments and detecting several mistakes in the earlier version of the proof of Theorem \ref{t3.1}. We would  also like to thank the referees for making several useful comments.

\end{document}